\documentclass{article}

\usepackage{amsmath, amsthm, amsfonts, amssymb, cite}
\usepackage{wrapfig}
\usepackage{stackrel}
\usepackage{color}
\usepackage{graphicx}
\usepackage{float}
\usepackage{enumitem}
\usepackage{mathtools}
\usepackage{multicol}

\newtheorem{theorem}{Theorem}
\newtheorem{lemma}[theorem]{Lemma}
\newtheorem{prop}[theorem]{Proposition}

\newcommand{\lt}{\left}
\newcommand{\rt}{\right}
\newcommand{\bpm}{\begin{pmatrix}}
\newcommand{\epm}{\end{pmatrix}}
\newcommand{\bsm}{\lt(\begin{smallmatrix}}
\newcommand{\esm}{\end{smallmatrix}\rt)}
\newcommand{\beq}{\begin{equation}}
\newcommand{\eeq}{\end{equation}}
\newcommand{\bmat}{\begin{matrix*}}
\newcommand{\emat}{\end{matrix*}}

\renewcommand{\d}{\mathrm{d}}

\newcommand{\Z}{\ensuremath{\mathbb{Z}}}
\newcommand{\N}{\ensuremath{\mathbb{N}}}

\newcommand{\R}{\ensuremath{\mathbb{R}}}
\newcommand{\C}{\ensuremath{\mathbb{C}}}

\renewcommand{\H}{\ensuremath{\mathbb{H}}}

\newcommand{\vep}{\varepsilon}

\DeclareMathOperator{\SL}{SL}
\DeclareMathOperator*{\Res}{Res}

\newcommand{\aquad}{\qquad\qquad}
\newcommand{\bquad}{\aquad\aquad}

\renewcommand{\th}{\textsuperscript{th}}

\newcommand{\hf}{\frac{1}{2}}
\newcommand{\qtr}{\frac{1}{4}}

\renewcommand{\Re}{\operatorname{Re}}
\renewcommand{\Im}{\operatorname{Im}}

\usepackage[usenames,dvipsnames]{xcolor}

\title{Opposite Sign Kloosterman Sum Zeta Function}
\author{Eren Mehmet K\i ral}

\usepackage{mathtools}

\newcommand{\msp}{\mspace{-2mu}}

\begin{document}
     \maketitle
%NEW SECTION
%NEW SECTION
%NEW SECTION
%NEW SECTION

\begin{abstract}
	We study the meromorphic continuation and the spectral expansion of the oppposite sign Kloosterman sum zeta function,
	$$(2\pi \sqrt{mn})^{2s-1}\sum_{\ell=1}^\infty \frac{S(m,-n,\ell)}{\ell^{2s}}$$ for $m,n$ positive integers, to all $s \in \C$. There are poles of the function
	corresponding to zeros of the Riemann zeta function and the spectral parameters of Maass forms. The analytic properties of this function are rather 
	delicate. It turns out that the spectral expansion of the zeta function converges only in a left half-plane, disjoint from the region of absolute
	convergence of the Dirichlet series, even though they both are analytic expressions of the same meromorphic function on the entire complex plane. 
\end{abstract}

\section{Introduction} \label{sec:intro}

        Kloosterman sums play a central role in the spectral theory of modular forms, especially in regard to the Kuznetsov trace formula. One approach in studying 
        Kloosterman sums and related spectral sums is to study the Dirichlet series,
        \begin{equation}\label{eq:zetaDirichlet}
		Z_{m,n}(s) = (2\pi \sqrt{|mn|})^{2s-1} \sum_{\ell =1}^\infty \frac{S(m,n,\ell)}{\ell^{2s}},
        \end{equation}
        where 
        \[
		S(m,n,\ell) = \sum_{\substack{a,\overline{a} \mod \ell\\ a\overline{a} \equiv 1 \mod \ell}} e^{2\pi i \left(an + \overline{a}m \right)/\ell}
        \]
	is the Kloosterman sum associated to the group $\SL(2,\Z)$. In studying this Dirichlet series one may realize it as a Petersson inner product 
        of two real analytic Poincare series; this approach was succinctly carried out in \cite{motohashi1995kloosterman}. 
        
        Another form of the Kuznetsov trace formula is a spectral expansion for the function
        \[
		K(m,n,\phi) = \sum_{\ell=1}^\infty \frac{S(m,n,\ell)}{\ell} \phi\lt(\frac{4\pi \sqrt{|mn|}}{\ell}\rt)
        \]
        with $\phi$ a smooth function having sufficient decay conditions. One may relate one to the other via the Mellin transform formula
        \begin{equation}\label{eq:mellin}
		K(m,n,\phi) = \frac{1}{2\pi i} \int_{(1)} Z_{m,n}(s) \phi^*(s) \d s
        \end{equation}
        where $\phi^*$ is related to the Mellin transform of $\phi$;
        \[
		\phi^*(s) = \int_0^\infty \phi(x) \left( \frac{x}{2} \right)^{-(2s-1)} \frac{\d x}{x}.
        \]

	Both forms of the Kuznetsov trace formula have been extensively studied by Motohashi, (see \cite{motohashi1997spectral} Lemma 2.5 and Theorem 2.3 for 
	spectral expansions of either version and (2.4.6) for the necessary decay conditions on the function $\phi$). However in the case of the zeta function
	$Z_{m,n}(s)$, this was only done when $mn>0$. The case of $m$ and $n$ having opposite signs was left out. In fact, while deriving the opposite sign 
	Kuznetsov trace formula,  Motohashi makes the following remark 
	in his book: \emph{``Unfortunately we do not have an analogue of [the spectral expansion of Kloosterman sum zeta function] in our present situation, 
	which would make our problem easier''} while  trying to employ a variant of \eqref{eq:mellin} in case $mn<0$. Instead he obtains the Kuznetsov trace 
	formula involving the smooth function $\phi$ with recourse to the function
	\begin{equation}\label{eq:modifiedZetaDirichlet}
		\widetilde{Z}_{m,-n}(s) = (2\pi \sqrt{mn})^{2s-1} \sum_{\ell=1}^\infty \frac{S(m,-n;\ell)}{\ell^{2s}} \exp(-4\pi \sqrt{mn}/\ell)
	\end{equation}
	and its spectral expansion. We show that we can obtain a spectral expansion of $Z_{m,-n}(s)$ using that of $\widetilde{Z}_{m.-n}(s)$.
	
	The main goal of this note is to prove the following theorem.
	%theorem
	%theorem
	\begin{theorem}\label{thm:spectralExpansion}
		Let $m,n>0$. The function $Z_{m,-n}(s)$ defined in \eqref{eq:zetaDirichlet} has a meromorphic continuation to all $s\in\C$. 
		Let $R\geq 0 $ be an integer. Let $u_j$ be $L^2$-normalized Maass forms with eigenvalue $\qtr + t_j^2$ and let $\rho_j(n)$ be the $n$\th\ Fourier coefficient
		of $u_j$. For $s \in \C$ such that $\Re(s)<0$
		we have the following spectral expansion:
		\[
			Z_{m,-n} (s) = Z_d(s) + Z_c(s).
		\]
		Here 
		\[
			Z_d(s) = \hf \sum_{j=1}^\infty \overline{\rho_j(m)}\rho_j(-n)  \Gamma(s-\tfrac12 +it_j) \Gamma(s-\tfrac12-it_j),
		\]
		is the discrete spectrum and if $-R - \hf < \Re(s) < -R + \hf$, 
		\[
			Z_c(s) = I(s) + \sum_{r = 0}^R R^-_r(s) 
		\]
		is the continuous spectrum. The continuous spectrum consists of the integral
		\[
			I(s) = \frac{1}{2\pi} \int_{-\infty }^{\infty } \frac{\sigma_{2it}(m) \sigma_{2it}(n)}{(mn)^{it} |\zeta(1 + 2it)|^2}  \cosh(\pi t)  
			\Gamma(s - \tfrac12 + it)\Gamma(s - \tfrac12 - it)\d t
		\]
		and the additional terms obtained as residues when analytically continuing the integral expression $I(s)$ succesively into the strips 
		$-r -\hf < \Re(s) < -r + \hf$, defined as
		\[
			R_r(s) = 2 \frac{(-1)^r}{r!} \Gamma(2s +r - 1)
			\frac{\sigma_{2s + 2r-1}(m)\sigma_{2s+ 2r-1}(n)}{(mn)^{s +r- \hf} \zeta^*(2s + 2r)\zeta^*(2-2s-2r)}.
		\]
		In the above notation the values $m,-n$ have been supressed. We may write $Z_d = Z_{m,-n}^d$ if we want to make the dependence explicit. We may also write
		$Z_d = Z_d^-$ or $R_r = R_r^-$ to indicate that the indices $m$ and $-n$ are of opposite sign.

		If $\Re(s) = -R + \hf$ then the continuous spectrum changes slightly,
		\[
			Z_{m,-n}(s) = Z_d(s) + I(s,\mathcal{C}) + \sideset{}{'} \sum_{r = 1}^{R} R_r(s).
		\]
		Here $\sum\nolimits'$ indicates that the $r = R$ term in the sum is halved, and given any contour of integration $\Omega$ from 
		$-i\infty$ to $i\infty$,
		\[
			I(s,\Omega) = \frac{1}{2\pi i}\int_{\Omega } \frac{\sigma_{2u}(m) \sigma_{2u}(n) \cos(\pi u) }{(mn)^{u} \zeta(1 + 2u)\zeta(1-2u)}
			\Gamma (s -\tfrac 12 + u) \Gamma(s - \tfrac 12 - u) \d u ,
		\]
		as long as the contour does not contain any poles of the integrand. The contour $\mathcal{C}$ is given 
		as $\Re(u) = 1/(20 \log |\Im(u) + 2|)$, and finally $\zeta^*(s) = \zeta(s)\Gamma(\frac{s}{2}) \pi^{-\frac{s}{2}}$ is the completed Riemann zeta function.
		
		Note that if $(0)$ is the path $\Re(s) = 0$, then $I(s, (0)) = I(s)$.
	\end{theorem}	
	%theorem ends
	%theorem ends

	Notice that the sum over the discrete spectrum in $Z_d(s)$ converges absolutely for $\Re(s)<0$, and the integral in the continuous spectrum converges for $\Re(s)<1$. This does not overlap with 
	the region of absolute convergence of the Dirichlet series \eqref{eq:zetaDirichlet}, and therefore any attempts at proving such a spectral expansion by trying to 
	insert a ``correct'' test function into the formula (2.5.6) of \cite{motohashi1997spectral} would fail, because that would have found a spectral expansion for 
	the Dirichlet series, and not its analytic continuation. This is the gist of the subtleties in dealing with $Z_{m,-n}(s)$, and is likely the reason that it is 
	left out of the literature.
	
	This phenomenon of a Dirichlet series converging on a right half plane and its spectral expansion converging slowly on a disjoint left half plane was
	also encountered in \cite{hoffstein2011multiple} with a shifted convolution series $\sum_{n=1}^\infty a(n) \overline{b(n+h)} n^{-s}$, where $a(n), b(n)$ are Fourier
	coefficients of modular forms. These two occurences are related; the Fourier coefficients of the Poincare series constructed by Jeffrey Hoffstein in order to
	obtain this Dirichlet series contains the Kloosterman sums $S(m,-h,c)$.

	Even though the two regions are disconnected, there is a separate expression, involving the auxilliary function $\widetilde{Z}_{m,n}(s)$, which ensures meromorphic
	continuation of the opposite sign Kloosterman sum zeta function to the entire complex plane. See the end of Section \ref{sec:linearCombinations}. We then prove the
	spectral expansion of $Z_{m,-n}(s)$ in Section \ref{sec:kuznetsovSpectralExpansions}.
	
	The expression for the discrete spectrum suggests that there should be poles at $s = \hf + it_j$ arising from the gamma factors in $Z_d(s)$, as well as poles of $R_0^-(s)$
	at $s = \rho/2$ where $\rho$ is a critical zero of the Riemann zeta function. Unfortunately these poles are not in the region of convergence of the spectral expansion, 
	nor in the region of absolute convergence of the Dirichlet series expression. Yet in Section \ref{sec:polesResidues} we show that there are poles where there should be, 
	and explicitly write the residues. The residues at these poles equal the apparent residues of the spectral expansion of $Z_{m,-n}(s)$.

	In any application it will be useful to have a bound of $Z_{m,-n}(s)$ in terms of its $mn$ and $\Im(s)$ dependence. We establish an estimation in Section 
	\ref{sec:growth}.
	
	As an easy application of the analytic properties of $Z_{m,-n}(s)$ we find a bound for sum of opposite sign Kloosterman sums, which is uniform in all the parameters.
	\begin{theorem} \label{thm:sharpCutoff}
		Let $m,n >0$, $\vep>0$ and $X \gg1$ a large quantity. One has the bound
		\[
			\sum_{\ell \leq X} \frac{S(m,-n,\ell) }{\ell} \ll_\vep X^{\frac{1}{6} +\vep}  ((m,n)^\vep + (mn)^\theta)  + X^\vep (mn)^{\qtr + \vep },
		\]
		where the implied constant only depends on $\vep>0$, and $\theta$ is the best progress towards the Ramanujan-Petersson conjecture. We may take $\theta = 7/64$ 
		according to the Kim-Sarnak result.
	\end{theorem}
	
	This sum has been studied in \cite{sarnak2009linnik} for the case $mn>0$. Sarnak and Tsimerman obtained strong estimates for this sum using a careful study of uniform
	asymptotic expansions of the Bessel functions which arise in Kuznetsov trace formula. Possibly their methods could be adapted to obtain a better bound in the opposite
	sign case as well.

	Let us also record the spectral expansion of the same sign Kloosterman sum zeta function. The following spectral expansion on the Kloosterman sum 
	zeta function is due to Kuznetsov and a reworking of it can be found in Motohashi's paper \cite{motohashi1995kloosterman}. We also include the expression of the 
	meromorphic continuation to all of $\C$ for completeness.

	%theorem
	%theorem	
	\begin{theorem}[Kuznetsov]	\label{thm:kuznetsovTrace}
		Define
		\[
			Z_{m,n}(s) = (2\pi\sqrt{mn})^{2s-1}\sum_{\ell=1}^\infty \frac{S(m,n;\ell)}{\ell^{2s}} 
		\]
		to be the Kloostermann sum zeta function. Then if $\Re(s) >\hf$ one has the following spectral expansion for any positive integer $m$ and $n$:
		\begin{align*}
			Z_{m,n}(s)&= \hf\sin(\pi s) \sum_{j=1}^\infty \frac{\rho_j(n)\rho_j(m)}{\cosh(\pi t_j)}\Gamma\lt(s-\tfrac12 +it_j\rt)\Gamma\lt(s-\tfrac12 -it_j\rt)\\
			&+\frac{1}{2\pi}\sin(\pi s)\int_{-\infty}^\infty \frac{\sigma_{2it}(m)\sigma_{2it}(n)}{(mn)^{it}|\zeta(1+2it)|^2}
			\Gamma\lt(s-\tfrac12 + it\rt)\Gamma\lt(s-\tfrac12-it\rt)\d t\\
			&+\sum_{k=1}^\infty p_{m,n}(k) \frac{\Gamma(k-1 + s)}{\Gamma(k+1-s)} - \frac{1}{2\pi}\delta_{m,n} \frac{\Gamma(s)}{\Gamma(1-s)}
		\end{align*}
		where 
		\begin{align*}
			p_{m,n}(k) &= (2k-1)\sum_{\ell=1}^\infty \frac{S(m,n;\ell)}{\ell}J_{2k-1}\lt(\frac{4\pi \sqrt{mn}}{\ell}\rt)\\
			&=(-1)^k\frac{\Gamma(2k)}{(4\pi\sqrt{mn})^{2k-1}}\sum_{f \in \mathcal{F}_{2k}}\frac{a_f(m)a_f(n)}{\langle f, f\rangle}  
			- \frac{(-1)^k}{2\pi}\delta_{m,n}.
		\end{align*} 
		Here $\mathcal{F}_{2k}$ is a basis of cuspidal Hecke-forms of weight $2k$ on $\SL(2,\Z)$, and $a_f(n)$ is the $n$\th\ Fourier coefficient of $f$. In order to
		refer to this symbol easily let us call the first second and third lines of the above formula by $Z_d^+(s) = Z_{m,n}^d(s), Z_c^+(s) = Z_{m,n}^+(s)$ and 
		$Z_{hr}(s) = Z_{m,n}^{hr}(s)$ respectively. 
		
		This spectral expansion has a meromorphic continuation to all $s \in \C$. However the spectral expansion is not given by the same formula as above. The 
		integral has an integrand which has a pole if $\Re(s) = \hf- r$ for $r\in \N$. We need to add correctional terms in order to obtain the meromorphic continuation.
		If $R\geq 0$ is an integer and $s$ lies in the region $R - \hf < \Re(s) < R + \hf$, then additional terms $\sum_{r =0}^R R_r^+(s)$ are included, where
		\[
			R_r^+(s)= (-1)^r R_r^-(s).
			%2\sin(\pi s) \frac{(-1)^r}{r!} \frac{\sigma_{2s + 2r - 1}(m) \sigma_{2s + 2r - 1} (n) }{(mn)^{s + r -\hf}\zeta (2s + 2r) \zeta(2 -2s -2r)} \Gamma(2s -1 +r).
		\]
		Finally in case $\Re(s) = \hf-R$ for some positive integer $R$, the integral in the ``continuous spectrum'' is replaced by the following path integeral,
		\[
			\frac{\sin(\pi s)}{2\pi i}\int_{\mathcal{C}}\frac{\sigma_{2u}(m)\sigma_{2u}(n)}{(mn)^{u}\zeta(1+2u)\zeta(1-2u)}
			\Gamma\lt(s-\tfrac12 + u\rt)\Gamma\lt(s-\tfrac12-u\rt)\d t,
		\]
		where the contour $\mathcal{C}$ is as above and and we add half of the $r = R$ term.
	\end{theorem}
	
	In obtaining the additional terms $R_r^+$ we have used the formula,
	\[
		\sin(\pi s) = (-1)^r \sin(\pi(s + r)) = \frac{(-1)^r \pi }{\Gamma(s + r) \Gamma(1-s-r)}
	\]
	changing the Riemann zeta functions in the denominator to completed ones. 
	%theorem ends
	%theorem ends

%NEW SUBSECTION
%NEW SUBSECTION
%NEW SUBSECTION
%NEW SUBSECTION

	\section{Expressing $Z_{m,-n}(s)$ in terms of $\widetilde{Z}_{m,-n}(s)$ and vice versa}\label{sec:linearCombinations}
	
		The function $\widetilde{Z}_{m,-n}(s)$ contains all the necessary information about the un-modified Kloosterman sum zeta function $Z_{m,n}(s)$. In fact we are going
		to extract the latter from a summation of the former. Let us express the exponential factor as a Taylor series around $0$.
	
		%proposition
		%proposition
		\begin{prop}\label{prop:ZTildeExpansion}
			For $\Re(s) >3/4$, we have the equality
			\[
				\widetilde{Z}_{m,-n}(s) = \sum_{k=0}^\infty \frac{(-2)^k}{k!} Z_{m,-n} (s + k/2) .
			\]
		\end{prop}
		%proposition ends
		%proposition ends
	
		%proof
		%proof
		\begin{proof}
			Expand the exponential function in a Taylor series and exchange orders of summation.
			\begin{align*}
				\widetilde{Z}_{m,-n}(s) &= (2\pi \sqrt{mn})^{2s-1} \sum_{\ell=1}^\infty \frac{S(m,-n;\ell)}{\ell^{2s}} \exp(-4\pi \sqrt{mn}/\ell) \\
				&=(2\pi \sqrt{mn})^{2s-1} \sum_{\ell=1}^\infty \frac{S(m,-n;\ell)}{\ell^{2s}} \sum_{k=0}^\infty \frac{1}{k!} \lt(-\frac{4\pi \sqrt{mn}}\ell \rt)^k\\
				&= \sum_{k=0}^\infty \frac{(-2)^k}{k!} (2\pi \sqrt{mn})^{2s + k -1} \sum_{\ell=1}^\infty \frac{S(m,-n;\ell)}{\ell^{2s + k}} \\
				&= \sum_{k=0}^\infty \frac{(-2)^k}{k!} Z_{m,-n}\lt(s + \frac{k}{2} \rt).\qedhere
			\end{align*}
		\end{proof}
		%proof ends
		%proof ends
	
		Now let us write $\widetilde{Z}_{m,-n}(s + k/2)$ in terms of $Z_{m,-n}$'s. We find a linear combination of the former that yields $Z_{m,-n}(s)$. 
		\[
			\begin{matrix}
				\widetilde{Z}_{m,-n}(s) \mspace{-4mu} & = & \mspace{-4mu} Z_{m,-n}(s) \mspace{-4mu}& -& \mspace{-4mu}\frac{2^1}{1!} Z_{m,-n}(s + \tfrac12)
				\mspace{-4mu}& +&\mspace{-4mu} \frac{2^2}{2!}Z_{m,-n} (s + 1) \mspace{-4mu}& - &\mspace{-4mu}\cdots \\
				\widetilde{Z}_{m,-n}(s+\tfrac12)\mspace{-4mu} & = & \mspace{-4mu}\mspace{-4mu} & & \mspace{-4mu} Z_{m,-n}(s + \tfrac12)
				\mspace{-4mu}& -& \mspace{-4mu}\frac{2^1}{1!}Z_{m,-n} (s + 1) \mspace{-4mu}&  + & \mspace{-4mu}  \cdots\\
				\widetilde{Z}_{m,-n}(s+1) \mspace{-4mu}& = &\mspace{-4mu} \mspace{-4mu} & &\mspace{-4mu} \mspace{-4mu} & & \mspace{-4mu}Z_{m,-n} (s + 1)
				\mspace{-4mu} & - & \mspace{-4mu}\cdots
			\end{matrix} 
		\]
	
		Notice that the combination $\widetilde{Z}_{m,-n}(s) + 2\widetilde{Z}_{m,-n}(s+\tfrac12)$ cancels the second column. The third column is also killed with the 
		linear combination $\widetilde{Z}_{m,-n}(s) + 2\widetilde{Z}_{m,-n}(s+\tfrac12) + 2^2/2! \widetilde{Z}_{m,-n}(s+1)$.
		%proposition
		%proposition
		\begin{prop}\label{prop:linearCombination}
			For $\Re(s)>3/4$, we have the equality, 
			\begin{equation}\label{eq:linearCombination}
				Z_{m,-n}(s) = \sum_{k=0}^\infty \frac{2^k}{k!} \widetilde{Z}_{m,-n}\lt(s + \tfrac k2 \rt).
			\end{equation}
		\end{prop}
		%proposition ends
		%proposition ends
	
		%proof
		%proof
		\begin{proof}
			We express the sum \eqref{eq:linearCombination} in terms of the columns (as we have written above) in their expansion in terms of the $Z_{m,n}$ functions:
			\[
				\sum_{i=0}^\infty \frac{2^i}{i!} \widetilde{Z}_{m,-n}\lt(s + \tfrac i2 \rt) = Z_{m,-n} (s) + \sum_{k=1}^\infty Z_{m,n}\lt(s + \tfrac k2\rt) 
				\sum_{i+j = k} \frac{(2)^{i}}{i!}\frac{(-2)^{j}}{j!}.
			\]
			The innermost summation simply vanishes:
			\[
				2^k \sum_{i + j = k} \frac{(-1)^i}{i!j!} = \frac{2^k}{k!} \sum_{i + j= k} \frac{k!}{i!j!} (-1)^i = \frac{2^k}{k!}(1-1)^k = 0.
			\]
			This yields the formula.
		\end{proof}
		%proof ends	
		%proof ends

		We have expressed the function $Z_{m,-n}(s)$ as an infinite linear combination of $\widetilde{Z}_{m,-n}(s)$. Motohashi, in \cite{motohashi1997spectral} gives a spectral 
		expansion of this latter modified Kloos{\-}terman sum zeta function for $\Re(s) >\frac 34$, as follows:
		\begin{align}\label{eq:modifiedZetaSpectralExpansion}
			\widetilde{Z}_{m,-n}(s) &= \sum_{j=1}^\infty \overline{\rho_j(m)}\rho_j(-n)  2^{3-4s} \sqrt{\pi} \frac{\Gamma(2s-1 + 2it_j)\Gamma(2s-1-2it_j)}{\Gamma(2s-\hf)}\\
			&+ \frac{2^{3-4s}}{\sqrt{\pi}} \int\limits_{-\infty }^{\infty } \frac{\sigma_{2it}(m) \sigma_{2it}(n) \cosh(\pi t)}{(mn)^{it} |\zeta(1 + 2it)|^2}  
			\frac{\Gamma(2s - 1 + 2it) \Gamma(2s-1 - 2it) }{\Gamma(2s-\hf)} \d t.\notag
		\end{align}
	
		This spectral expansion gives an analytic continuation of $\widetilde{Z}_{m,-n}(s)$, initially to $\Re(s)>\hf$, and then a meromorphic continuation to all $s\in \C$. 
		Notice, that although the infinite summation and the integral converges for all $s\in \C$ not on a pole of the summands or the integrand, this does not mean that 
		the analytic continuation of $\widetilde{Z}_{m,-n}(s)$ is given by the same formula as in \eqref{eq:modifiedZetaSpectralExpansion}. The poles of the integrand of the continuous spectrum separate 
		the regions where $s$ is allowed to be in. So the analytic continuation of the expression will be given by another formula. That is done in Theorem 
		\ref{thm:modifiedZetaSpectralExpansion} below. 
	
		Before that, let us see that \eqref{eq:linearCombination} gives us a meromorphic continuation to all $s \in \C$. 
		%proof of analytic continuation
		%proof of analytic continuation
		\begin{proof}[Proof of meromorphic continuation of $Z_{m,-n}(s)$]
			For $\Re(s) >1$, the function $\widetilde{Z}_{m,-n}(s)$ is given by its Dirichlet series \eqref{eq:modifiedZetaDirichlet}, 
			and we can bound it by $\zeta(3/2) <3$. Therefore the summation
			\[
				\sum_{k = 0}^\infty \frac{2^k}{k!} \widetilde{Z}_{m,-n} \lt(s + \frac k2\rt) 
			\]
			converges absolutely. Except for finitely many terms, the terms of the summation can be bounded by 
			$2\sum_{k \text{ s.t. } \Re(s + k/2)>1} \frac{2^k}{k!} <\infty$. In fact this shows that the convergence is uniform in compact sets of 
			$s \in \C$ away from the poles of $\widetilde{Z}_{m,-n}(s + k/2)$. 

			Furthermore for $\Re(s) >3/4$, i.e. in the region of absolute convergence of these Dirichlet series, we have equality. So the infinite summation over $k$ 
			gives a meromorphic continuation of $Z_{m,-n}(s)$ to all $s$. 
		\end{proof}
		%proof of analytic continuation ends
		%proof of analytic continuation ends
	
%NEW SUBSECTION
%NEW SUBSECTION
%NEW SUBSECTION
%NEW SUBSECTION

	\section{Spectral Expansions}\label{sec:kuznetsovSpectralExpansions}

		The above proof used that $\widetilde{Z}_{m,-n}$ had meromorphic continuation to all $s \in \C$. We write below the expression giving the function in various domains.
		
		%theorem
		%theorem
		\begin{theorem} \label{thm:modifiedZetaSpectralExpansion}
			Let 
			\[
				\widetilde{Z}_d(s) = \sum_{j=1}^\infty \overline{\rho_j(m)} \rho_j(-n) 2^{3-4s}\sqrt{\pi} 
				\frac{\Gamma(2s - 1 +2it_j)\Gamma(2s-1 -2it_j)}{\Gamma(2s-\tfrac12)},
			\]
			and
			\[
				\widetilde{I}(s,\Omega) \mspace{-2mu} = \mspace{-2mu} \frac{2^{4-4s}\sqrt{\pi}}{2\pi i } \int_{\Omega} 
				\frac{\sigma_{2u}(m) \sigma_{2u}(n)\cos(\pi u)}{(mn)^u \zeta(1 + 2u) \zeta(1-2u)}   
				\frac{\Gamma(2s \mspace{-3mu}-\mspace{-3mu} 1\mspace{-3mu} +\mspace{-3mu}2u) \Gamma(2s \mspace{-3mu}-\mspace{-3mu} 1\mspace{-3mu} -\mspace{-3mu} 2u)}
				{\Gamma(2s - \tfrac12)}\d u.
			\]
			where $\Omega$ is a path from $-i\infty$ to $i\infty$ on $\C$. Let $(0)$ be the vertical path $\Re(u) = 0$, and $\mathcal{C}$ be the path
			$\Re(u) = \frac{1}{20 \log|\Im(u) + 2|}$. This region is chosen so that there are no zeros of $\zeta(1-2u)$ between the curves $(0)$ and $\mathcal{C}$. 
			Let us also call, for brevity, $\widetilde{I}(s,(0)) = \widetilde{I}(s)$. 
		
			The spectral expansion of $\widetilde{Z}_{m,-n}(s)$ in the region $\Re(s)\geq \hf$ is given by 
			\[
				\widetilde{Z}_{m,-n}(s) = \widetilde{Z}_d(s) + \widetilde{I}(s,(0)).
			\] 
			Let us call
			\begin{equation}\label{eq:modifiedZetaResidualTerm}
				\widetilde{R}_r(s) :=\frac{(-1)^r}{r!}  \frac{2^{3-4s}\sqrt{\pi} \sigma_{2s-1+r}(m)\sigma_{2s-1+r}(n)}{(mn)^{s-\hf+\frac r2}\zeta^*(2s+r) \zeta^*(2-2s-r)}
				\frac{\Gamma(4s-2+r)}{\Gamma(2s -\tfrac12)}.
			\end{equation}
			Then for $s$ such that $-\frac{R}{2}<\Re(s)<-\frac{R}{2} +\hf$, 
			\begin{equation}\label{eq:modifiedZetaSpectralExpansionCompact}
				\widetilde{Z}_{m,-n}(s) = \widetilde{Z}_d(s) + \widetilde{I}(s) + \sum_{r=0}^R 2\widetilde{R}_r(s).
			\end{equation}
			If $\Re(s) = -\frac{R}{2} + \hf$,
			\begin{equation} \label{eq:modifiedZetaSpectralExpansionOnTheLine}
				\widetilde{Z}_{m,-n}(s) = \widetilde{Z}_d(s) + \widetilde{I}_c(s,\mathcal{C}) + \sum_{r=0}^{R-1} 2\widetilde{R}_r(s) + \widetilde{R}_R(s).
			\end{equation}
			The terms other than the discrete spectrum will be denoted by $\widetilde{Z}_c(s)$ and is called the continuous part of $\widetilde{Z}_{m,-n}(s)$. 
		\end{theorem}
		%theorem ends
		%theorem ends
	
		%proof
		%proof
		\begin{proof}
			The spectral expansion \eqref{eq:modifiedZetaSpectralExpansion}, given in \cite{motohashi1997spectral}, for $\widetilde{Z}_{m,-n}(s)$ actually converges 
			in $\Re(s)>\hf$. In fact the sum converges as long as $s$ does not equal one of $\frac{r}{2} + \hf \pm it_j$ for a nonnegative integer $r$ and spectral 
			parameter $t_j$, and the integral converges as long as $s$ does not lie on the lines $\Re(s) = \hf - \frac{r}{2}$, $r\geq 0$. 
		
			However, these regions of $s$ are separated by the poles of the integrand as $u$ varies in $i\R$, and hence the same expression 
			$\widetilde{Z}_d(s) + \widetilde{I}(s)$ is not the analytic continuation of itself for $\Re(s)>\hf$ and $0<\Re(s)<\hf$.
		
			Now let $r \geq 0$ and $s$ lie between the curves $(-\frac{r}{2} + \hf)$ and $\mathcal{C} - \frac{r}{2} + \hf$. We may express,
			\begin{align*}
				\widetilde{I}(s)\msp\msp &= \msp \frac{(-1)^r}{r!}  \frac{2^{3-4s}\sqrt{\pi}\sigma_{2s-1+r}(m)\sigma_{2s-1+r}(n) 
				\cos(\pi (s\msp-\msp \tfrac12 \msp + \msp\tfrac r2))} {(mn)^{s-\hf+\frac r2}\zeta(2s+r) \zeta(2-2s-r)}\frac{\Gamma(4s\msp-\msp 2\msp +\msp r)}
				{\Gamma(2s -\tfrac12)} \\
				&\phantom{=} + \widetilde{I}_c(s,\mathcal{C}) \\ 
				&= \frac{(-1)^r}{r!}  \frac{2^{3-4s}\sqrt{\pi} \sigma_{2s-1+r}(m)\sigma_{2s-1+r}(n)}{(mn)^{s-\hf+\frac r2}\zeta^*(2s+r) \zeta^*(2-2s-r)} 
				\frac{\Gamma(4s-2+r)}{\Gamma(2s -\tfrac12)}\mspace{-3mu} +\mspace{-3mu} \widetilde{I}(s,\mathcal{C}) 
			\end{align*}
			where in the last line we have turned the cosine function into sine and that into a product of reciprocal gamma functions. Notice that 
			$\widetilde{I}(s,\mathcal{C})$ is an analytic function in the region where $s$ is between $-\mathcal{C} - \frac r2 + \hf$ and $\mathcal{C} - \frac{r}2 + \hf$. 
			Thus we get an analytic continuation of $\widetilde{I}_c(s)$ into this domain, which includes the line $\Re(s) = -\frac r2 + \hf$. Now if we assume 
			that $s$ is to the left of that line, when we move the line of integration back to $(0)$, we pass over the pole at $2s -1 +2u = -r$, 
			the residue is the same.
			\begin{align*}
				  \widetilde{I}(s) &= \frac{(-1)^r}{r!}  \frac{2^{3-4s}\msp \sqrt{\pi} \sigma_{2s-1+r}(m)\sigma_{2s-1+r}(n)}
				  {(mn)^{s-\hf+\frac r2}\zeta^*(2s+r) \zeta^*(2-2s-r)} \frac{\Gamma(4s\mspace{-2mu}-\mspace{-2mu}2+r)}{\Gamma(2s \mspace{-3mu}-\mspace{-3mu}\tfrac12)} 
				  \mspace{-3mu}+\mspace{-3mu} \widetilde{I}(s,\mathcal{C}) \\
				  &=2\frac{(-1)^r}{r!}  \frac{2^{3-4s}\sqrt{\pi}\sigma_{2s-1+r}(m)\sigma_{2s-1+r}(n)}{(mn)^{s-\hf+\frac r2}
				  \zeta^*(2s\mspace{-2mu}+\mspace{-2mu}r)  \zeta^*(2\mspace{-2mu}-\mspace{-2mu}2s\mspace{-2mu}-\mspace{-2mu}r)} 
				  \frac{\Gamma(4s\mspace{-3mu}-\mspace{-3mu}2\mspace{-3mu}+\mspace{-3mu}r)}{\Gamma(2s -\tfrac12)} \mspace{-3mu}
				  + \mspace{-3mu} \widetilde{I}(s).\qedhere
			\end{align*}
		\end{proof}
		%proof ends
		%proof ends
	
		We are going to plug in this spectral expansion into \eqref{eq:linearCombination} and obtain a proof of Theorem \ref{thm:spectralExpansion}.
	
		%proof
		%proof 
		\begin{proof}[Proof of Theorem \ref{thm:spectralExpansion}]
			We take the infinite linear combination \eqref{eq:linearCombination} of $\widetilde{Z}_{m,-n}(s + k/2)$ expressing the function $Z_{m,-n}(s)$. This summation
			converges for all $s \in \C$ except for when $s + \frac{k}{2}$ lies on a pole of $\widetilde{Z}_{m,n}$ for some $k$. For each $\widetilde{Z}_{m,-n}(s + k/2)$ plug 
			in the spectral expansion in Theorem \ref{thm:modifiedZetaSpectralExpansion}. This will be a spectral expansion of $Z_{m,-n}(s)$. As it will turn out, 
			although \eqref{eq:linearCombination} converges for all $s \in \C$, the spectral expansion will only absolutely converge for $\Re(s)<0$. 
		
			Now assume that $r$ is a nonnegative integer and $-\frac{R}{2} < \Re(s) < -\frac{R}{2} +\hf$.
			\begin{align*}
				Z_{m,-n}(s) &= \sum_{k=0}^\infty  \frac{2^k}{k!} \widetilde{Z}_{m,-n}\lt(s + \tfrac k2 \rt) \displaybreak[0]\\
				&= \sum_{j=1}^\infty \overline{\rho_j(m)}\rho_j(-n)  2^{3-4s} \sqrt{\pi} \\
				&\bquad\times \sum_{k = 0}^\infty \frac{2^{-k}}{k!} \frac{\Gamma(2s + k-1 + 2it_j)\Gamma(2s + k -1-2it_j)}{\Gamma(2s + k -\hf)}\\
				&+ \frac{1}{\pi} \int_{-\infty }^{\infty } \frac{\sigma_{2it}(m) \sigma_{2it}(n)}{(mn)^{it} |\zeta(1 + 2it)|^2} 
				2^{3-4s} \sqrt{\pi} \cosh(\pi t)  \\
				&\bquad \times \sum_{k=0}^\infty \frac{2^{-k}}{k!}\frac{\Gamma(2s +k - 1 + 2it) \Gamma(2s + k -1 - 2it) }{\Gamma(2s+k -\hf)} \d t\\
				& + \sum_{k=0}^R \sum_{j=0}^{R-k} \frac{2^k}{k!} \widetilde{R}_j\lt(s + \frac{k}{2}\rt).
			\end{align*}	
		
			The summation over $k$ in the discrete and continuous spectra can be explicitly evaluated. Call the sum over these Gamma functions by
			$G(s)$.
			\begin{align*}
				G(s) &= \sum_{k=0}^\infty \frac{2^{-k}}{k!}\frac{\Gamma(2s +k - 1 + 2it) \Gamma(2s + k -1 - 2it) }{\Gamma(2s+k -\hf)} \\
				&= \frac{\Gamma(2s - 1 + 2it)\Gamma(2s - 1 - 2it)}{\Gamma(2s-\tfrac12)} {}_2F_1(2s-1+2it,2s-1-2it, 2s - \tfrac12, \tfrac12). 
			\end{align*}
			The Gauss hypergeometric function appears because the summation over $k$ is exactly its Taylor series expansion. where we have used the formula
			\begin{equation}\label{eq:hypergeometricSpecialValue}
				{}_2F_1\lt(a,b; \frac{a + b}{2} + \hf; \hf\rt)  
				= \frac{\sqrt{\pi} \Gamma\lt(\frac{a + b}{2} + \hf\rt)}{\Gamma\lt(\frac{a + 1}{2}\rt)\Gamma\lt(\frac{b+1}{2}\rt)}.
			\end{equation}		
			to evaluate the Gauss hypergeometric function at $1/2$, which can be found at \cite{GradsteynRhyzik}, \# 9.136.1.
			Legendre's duplication formula of the gamma function simplifies this expression to give
			\begin{align*}
				G(s) &=  \frac{\Gamma(2s - 1 + 2it)\Gamma(2s - 1 - 2it)}{\Gamma(2s-\tfrac12)} \sqrt{\pi} \frac{\Gamma(2s - \tfrac12)}{\Gamma(s+it)\Gamma(s-it)}\\
				&= \frac{2^{4s - 4}}{\sqrt{\pi}} \Gamma(s - \tfrac12 + it)\Gamma(s - \tfrac12 - it).
			\end{align*}
		
			Having established the summation over the discrete spectrum and the integral in the continuous spectrum we now turn our attention to 
			the additional residue terms that show themselves in studying the continuous spectrum with $\Re(s)<\hf$. We will make use of the following lemma.
			%lemma
			%lemma
			\begin{lemma}
			For $\widetilde{R}_j(s)$ defined as in \eqref{eq:modifiedZetaResidualTerm},
			\[
				\widetilde{R}_j(s) = \frac{(-1)^j}{j!} 4^j \frac{(2s-\tfrac12)_j}{(4s-2 + j)_j} \widetilde{R}_0\lt(s + \frac{j}{2}\rt),
			\]
			where $(z)_n = \Gamma(z + n)/\Gamma(z)$ is the Pochhammer symbol, or the rising factorial.
			\end{lemma}
			%lemma ends
			%lemma ends
	
			%proof of lemma
			%proof of lemma
			\begin{proof}
				We simply realize the shift in $s$ caused by the $j$ index, in the definition of $\widetilde{R}_j(s)$. 
				\begin{align*}
					\widetilde{R}_j(s) &= \frac{(-1)^j}{j!} \frac{2^{3-4s}\sigma_{2s-1+j}(m)\sigma_{2s-1 + j}(n)}{(mn)^{s - \hf + \frac j2} 
					\zeta^*(2s + j)\zeta^*(2-2s-j)} \frac{\Gamma(4s-2+j)}{\Gamma(2s-\tfrac12)}\\
					&= \frac{(-1)^j}{j!} 2^{2j} \frac{\Gamma(4s-2+j)}{\Gamma(4s-2+2j)}\frac{\Gamma(2s-\tfrac12 + j)}{\Gamma(2s - \tfrac12)} \\
					&\aquad \times \frac{2^{3-4(s + \frac j2)}\sigma_{2(s + \frac j2)-1}(m)\sigma_{2(s+ \frac j2)-1}(n)}{(mn)^{(s +\frac j2)- \hf} 
					\zeta^*(2(s + \tfrac j2))\zeta^*(2-2(s+\tfrac j2))} \frac{\Gamma(4(s + \tfrac j2) -2)}{\Gamma(2(s+\tfrac j2) - \tfrac12)}\\
					&= \frac{(-1)^j}{j!} 4^j \frac{(2s- \tfrac12)_j}{(4s -2 +j)_j} \widetilde{R}_0\lt(s + \frac j2\rt).\qedhere
				\end{align*}
			\end{proof}
			%proof of lemma ends
			%proof of lemma ends

			We also record the following equality for convenience,
			\begin{equation}\label{eq:shiftResidual}
				\widetilde{R}_j\lt(s + \tfrac{k}{2}\rt) = \frac{(-1)^j}{j!}4^j \frac{(2s +k- \tfrac 12)_j}{(4s + 2k-2 +j)_j} \widetilde{R}_0\lt(s + \frac{j + k}{2}\rt).
			\end{equation}
			
			Call $r = j + k$ and sum over $k$ and $j$,
			\[
				\sum_{k=0}^R \frac{2^{k}}{k!} \sum_{j=0}^{R-k} \widetilde{R}_j\lt(s +  \frac{k}{2}\rt)
				= \sum_{r = 0}^R \widetilde{R}_0\lt(s +  \frac{r}{2}\rt) \sum_{k + j= r} \frac{2^{k}}{k!} \frac{(-1)^j}{j!}4^j 
				\frac{(2s +k- \tfrac 12)_j}{(4s + 2k-2 +j)_j}.
			\]
		
			Denote each summand of the sum over $r$ by $R_{\frac r2} (s)$, which is the analogue of $\widetilde{R}_r(s)$ in $Z_{m,-n}(s)$. 
			Now we look at the inner sum, i.e.\ the coefficient of $\widetilde{R}_0 (s + r/2)$ in this sum. Call it $G_r$, and 
			\begin{align*}
				G_r &= \sum_{k + j = r} \frac{2^k}{k!}\frac{(-1)^j4^j}{j!} \frac{(2s + k-\hf)_j}{(4s + 2k -2 + j)_j}\\
				& = \sum_{k + j = r} \frac{2^k}{k!}\frac{(-4)^j}{j!} \frac{\Gamma(2s + k - \hf + j)}{\Gamma(2s + k - \hf)} 
				\frac{\Gamma(4s + 2k -2 + j)}{\Gamma(4s + 2k - 2 + 2j)} \frac{\Gamma(2s + k-1 + j)}{\Gamma(2s + k - 1 + j)}\displaybreak[0]\\
				& = \frac{1}{\Gamma(2s  - 1 +r)} \sum_{k + j = r} \frac{2^k}{k!}\frac{(-4)^{r-k}}{(r-k)!} \sqrt{\pi} 2^{3 - 4s - 2r}  
				\frac{\Gamma(4s + 2k -2 + j)}{ \Gamma(2s + k - \hf) },
			\end{align*}
			where we have used the Legendre duplication formula. The sum over $k$ gives a polynomial which 
			can be viewed as a hypergeometric function.
			\begin{align*}
				G_r & = \frac{(-1)^r\sqrt{\pi} 2^{3 - 4s }}{r! \Gamma(2s  - 1 +r)} \sum_{k =0}^r \frac{2^{-k} (-1)^k r!}{k!(r-k)!}  
				\frac{\Gamma(4s -2  + r + k)}{ \Gamma(2s + k - \hf) }\displaybreak[0]\\
				& = \frac{(-1)^r\sqrt{\pi} 2^{3 - 4s }}{r! \Gamma(2s  - 1 +r)} \frac{\Gamma(4s - 2 + r)}{\Gamma(2s - \hf)}\sum_{k=0}^r 
				\frac{(-r)_k (4s -2 + r)_k}{(2s - \hf )_k} \frac{2^{-k}}{k!}\\
				& = \frac{(-1)^r\sqrt{\pi} 2^{3 - 4s }}{r! \Gamma(2s  - 1 +r)} \frac{\Gamma(4s - 2 + r)}{\Gamma(2s - \hf)}
				{}_2F_1(-r,4s - 2 + r; 2s - \tfrac12; \tfrac 12).
			\end{align*}
		
			We use \eqref{eq:hypergeometricSpecialValue} once again in evaluating the hypergeometric function. Since the hypergeometric 
			function vanishes for $r$ odd and therefore we put $r = 2r'$. Thus,
			\begin{align*}
				\sum_{r=0}^R \widetilde{R}_\frac{r}{2}(s) \mspace{-1mu}&= \sum_{r=0}^R \widetilde{R}_0\lt(s + \frac{r}{2}\rt)  
				\frac{(-1)^r\sqrt{\pi} 2^{3 - 4s }}{r! \Gamma(2s  - 1 +r)} 
				\frac{\Gamma(4s \mspace{-2mu}-\mspace{-2mu} 2\mspace{-2mu} + \mspace{-2mu}r)}{\Gamma(2s - \hf)}
				\frac{\sqrt{\pi}\Gamma\lt(2s - \hf\rt) }{\Gamma(\frac{1-r}{2}) \Gamma(2s - \hf  + \frac r2)} \\
				& = \sum_{r = 0}^R \widetilde{R}_0\lt(s + \frac{r}{2} \rt)\frac{(-1)^r\sqrt{\pi} 2^{r}\Gamma(2s - 1 + \frac r2)}{r!\Gamma(2s-1 +r)\Gamma(\frac{1-r}{2})}\\
				& = \sum_{r' = 1}^{\lfloor \frac{R}{2} \rfloor}\frac{(-1)^{r'}}{(r')!} \frac{\Gamma(2s - 1 + r')}{\Gamma(2s -1 + 2r')}   \widetilde{R}_0(s  + r').
			\end{align*}
		
			Finally substituting the expression for $\widetilde{R}_0$,
			\begin{align}
				\sum_{r=0}^{\lfloor \frac R2\rfloor} \widetilde{R}_r(s) &= \sum_{r=0}^{\lfloor \frac R2\rfloor}\frac{(-1)^r}{r!} 
				\frac{\Gamma(2s - 1 + r)}{\Gamma(2s -1+2r)}\frac{\Gamma(4(s + r) -2)}{\Gamma(2(s+r) - \tfrac12)} \label{eq:vanishResidualTerms} \\
				&\bquad\times\frac{2^{3-4(s+r)}\sqrt{\pi} \sigma_{2(s + r)-1}(m)\sigma_{2(s+ r)-1}(n)}{(mn)^{(s +r)- \hf} \zeta^*(2(s + r))\zeta^*(2-2(s+r))}
				\displaybreak[0] \notag\\
				&=\sum_{r=0}^{\lfloor \frac R2\rfloor} \frac{(-1)^r}{r!} \Gamma(2s +r - 1)
				\frac{\sigma_{2s + 2r-1}(m)\sigma_{2s+ 2r-1}(n)}{(mn)^{s +r- \hf} \zeta^*(2s + 2r)\zeta^*(2-2s-2r)}.\notag
			\end{align}
		  
			Since the odd terms vanish we see that the only additional terms come when we pass $\Re(s) = -R + \hf$ for a nonnegative integer $R$. 
			If $\Re(s) = - R + \hf$ then the treatment is identical with \eqref{eq:modifiedZetaSpectralExpansionCompact} replaced by 
			\eqref{eq:modifiedZetaSpectralExpansionOnTheLine}.
		\end{proof}
		%proof ends
		%proof ends

%NEW SECTION
%NEW SECTION
	
	\section{Poles and Residues}\label{sec:polesResidues}
	
		The opposite sign Kloosterman sum zeta function has poles in the region $0<\Re(s)< 1$ where the discrete part of the spectral expansion does not converge. Yet the
		expression \eqref{eq:linearCombination} in Proposition \ref{prop:linearCombination} allows us to identify those poles as well. 

		%proposition
		%proposition
		\begin{prop} \label{prop:criticalPoles}
			For $m,n >0$, the function $Z_{m,-n}(s)$ is analytic in the region $\Re(s) > \hf$, has simple poles at $s = \hf + it_j$ where $\qtr + t_j^2$ is the 
			$j$\th\ Laplace eigenvalue on $\SL(2,\Z)\backslash \H$, and poles at $\rho/2$ where $\zeta(\rho) = 0$. The order of the pole is determined by the order of 
			the zero of the Riemann zeta function. The residues of the poles at $s = \hf + it_j$ are given by
			\[
				\Res_{s = \hf + it_j} Z_{m,-n}(s) = \hf \lambda_j(m) \lambda_j(n) \rho_j(1) \overline{\rho_j(-1)} \Gamma(2it_j).
			\]
			For the pole coming from the continuous contribution in the spectral expansion it would be caused by $\zeta^*(2s)$ in the denominator i.e. the function
			\[
				Z_{m,-n}(s) - \frac{1}{\zeta^*(2s)} \lt(\frac{\sigma_{2s - 1}(m) \sigma_{2s -1}(n)}{(mn)^{s-\hf} \zeta^*(2s-1)}\Gamma(2s-1) \rt)
			\]
			is regular at $s = \frac{\rho}{2}$. Notice that the expression inside the parantheses is a holomorphic function of $s$ in $0\leq \Re(s) \leq \hf$, except
			at $s = 0$, and that simple pole is cancelled by the zero of $1/\zeta^*(2s)$ outside of the parantheses.
		\end{prop}
		%proposition ends
		%proposition ends
		
		%proof of proposition
		%proof of proposition
		\begin{proof}
			We look at the convergent expression
			\[
				Z_{m,-n}(s) = \sum_{k = 0}^\infty \frac{2^k}{k!} \widetilde{Z}_{m,-n}\lt(s + \frac k2\rt)
			\]
			given in Proposition \ref{prop:linearCombination}. The spectral expansion of $\widetilde{Z}_{m,-n}$ does not have any poles in the region $\Re(s)>\hf$, and 
			therefore in observing the poles of $Z_{m,-n}(s)$, we would only have to look at the $ k = 0$ term in the above infinite linear combination, i.e.
			\[
				\operatorname{Polar Part}_{s = \xi} Z_{m,-n}(s) = \operatorname{Polar Part}_{s = \xi} \widetilde{Z}_{m,-n}(s)
			\]
			for $s$ in the region $0< \Re(s) < 1$. 
			
			The rest of the proof is calculating the poles and residues of $\widetilde{Z}_d(s)$ and $\widetilde{R}_0(s)$ from the expressions given in Theorem 
			\ref{thm:modifiedZetaSpectralExpansion}. One needs to apply the duplication formula of the Gamma function in obtaining some of the simplification.
		\end{proof}
		%proof of proposition ends
		%proof of proposition ends
		
		The rest of the poles of $Z_{m,-n}(s)$ are in the region of absolute convergence for the spectral expansion of the function, hence we can use the expression 
		in Theorem \ref{thm:spectralExpansion}. Those poles are at the locations $\hf + it_j -r$ for $r \geq 1$ integers. The residues at these poles are given by
		\begin{equation} \label{eq:discretePoles}
			\Res_{s = \hf + it_j - r} Z_{m,-n}(s) = \hf \lambda_j(m) \lambda_j(n) \rho_j(1) \overline{\rho_j(-1)} \Gamma(2it_j - r).
		\end{equation}
		
		From the continuous spectrum, more precisely the terms $R_r(s)$ in Theorem \ref{thm:spectralExpansion}, there are going to be poles at $\frac{\rho}{2} -r$ where
		$r\geq 1$ is an integer, and where $\rho$ is a zero in the critical strip of the Riemann zeta function.
		
%NEW SECTION
%NEW SECTION
%NEW SECTION
%NEW SECTION

	\section{Growth in $m$, $n$, and $\Im(s)$.}\label{sec:growth}
	
		We will investigate the growth of $Z_{m,n}(s)$ in $m$, $n$ and the imaginary part of $s$. Call $\sigma = \Re(s)$. The aim of this section is to prove the
		following theorem
		%theorem
		%theorem
		\begin{theorem}\label{thm:growth}
			Let $m,n>0$ be positive integers and let $Z_{m,-n}(s)$ be the opposite sign Kloosterman sum zeta function. We have the following bound in the following regions.
			Let $\delta >0$;\\
			for $\sigma >\frac34$ and $\vep>0$ arbitrarily small
			\begin{equation} \label{eq:convergentBound}
				Z_{m,-n}(s) \ll_\vep (mn)^{\qtr + \vep}
			\end{equation}
			for $\hf + \delta <\sigma <\frac34$
			\begin{equation}\label{eq:largerThanHalfBound}
				Z_{m,-n}(s) \ll_\vep (mn)^\theta (1 + |t|)^\hf + (mn)^{\qtr + \vep},
			\end{equation}
			
			for $\sigma <0$,
			\begin{equation}
				Z_{m,-n}(s) \ll_s (mn)^{\hf - s},
			\end{equation}
			and for $-R - \hf < \sigma < -R$ for $R\geq 0$ an integer
			\begin{equation}
				Z_{m,-n}(s) \ll_\sigma \begin{cases} (mn)^\theta (1 + |t|)^{\sigma}  + (mn)^{\hf - \sigma }(1 + |t|)^{4\sigma -2} &\text{ if } R = 0,\\
					(mn)^\theta (1 + |t|)^\sigma + (mn)^{\hf -\sigma}(1 + |t|)^{\sigma - 2} &\text{ if } R\geq 1. \end{cases}
			\end{equation}
			Here $\theta$ is the best progress towards the Ramanujan conjecture.
			
			In the strips $-R <\sigma< -R + \hf$ for $R\geq 1$, the bound on the function $Z_{m,-n}(s)$ differs whether we assume the Riemann hypothesis (RH) or not,
			if $s = \sigma + it$ is chosen so that 
			\[
				\lt|s + R - \frac{\rho}{2} \rt| \gg \frac{1}{\log(1 + |t|) } 
			\]
			for all zeros $\rho$ of the Riemann zeta function then
			\begin{equation}
				Z(s) \ll \begin{cases} (mn)^{\hf - \sigma} (1 + |t|)^{\sigma - 2 + 1/2\pi} + (mn)^\theta (1 + |t|)^\sigma &\text{assuming the RH,}  \\
				(mn)^{\hf - \sigma} (1 + |t|)^{A \log\log (10 + |t|)} + (mn)^\theta (1 + |t|)^\sigma &\text{unconditionally.} \\ \end{cases}
			\end{equation}
			Finally for $0 < \sigma < \hf$,
			\begin{equation}\label{eq:firstCriticalBound}
				Z (s) \ll_\vep \begin{cases} (mn)^{\hf - \sigma}(1 + |t|)^{4\sigma - 2 + \frac{1}{2\pi} } + (mn)^\theta(1 + |t|)^\hf + 
				(mn)^{\qtr + \vep} &\text{RH}, \\
				(mn)^{\hf - \sigma}(1 + |t|)^{A \log \log (10 + |t|)} + (mn)^\theta (1 + |t|)^\hf +  (mn)^{\qtr + \vep}
				&\text{o.w.},\end{cases}
			\end{equation}
			under the same assumption as above.
			
			Employing Phragm\'en Lindel\"of Principle to the bounds \eqref{eq:convergentBound} and \eqref{eq:largerThanHalfBound} we obtain
			\begin{equation}\label{eq:phragmenLindelof}
				Z_{m,-n}(s) \ll_{\vep} (1 + |t|)^{\frac32 - 2\sigma + \vep} (mn)^{\qtr + \vep}
			\end{equation}
			in the region $\hf  < \sigma < \frac34$. 

		\end{theorem}
		%theorem ends
		%theorem ends
		
		As a last subsection we also state the growth of $Z_{m,n}(s)$ in terms of its $mn$ dependence for completeness.

		\subsection*{Region of Absolute Convergence of Dirichlet Series}
		
			Notice that when $\sigma >1$, we can trivially bound the Kloosterman sum in the Dirichlet series. Hence,
			\[
				|Z_{m,n}(s)| \leq (2\pi\sqrt{mn})^{2\sigma-1} \zeta(2\sigma -1) 
				\ll (mn)^{\sigma - \hf}.
			\]
			However we can also use Weil's bound,
			\[
				S(m,n,\ell) \leq \tau(\ell) (m,n,\ell)^\hf \ell^\hf
			\]
			where $\tau$ is the divisor function. 
			%lemma
			%lemma
			\begin{lemma}\label{lem:weilBound}
				In the region $\Re(s) =\sigma >\frac 34$, we have the bounds,
				\[
					Z_{m,n}(s) \ll_\vep (mn)^{\sigma-\hf}(m,n)^\vep \qquad \text{ and } \qquad \widetilde{Z}_{m,n}(s) \ll_\vep (mn)^{\sigma - \hf} (m,n)^\vep.
				\]
			\end{lemma}
			%lemma ends
			%lemma ends
			
			%proof of lemma
			%proof of lemma
			\begin{proof}
				Note that we can bound both of these functions with
				\[
					(2\pi \sqrt{mn})^{2\sigma -1} \sum_{\ell=1}^\infty |S(m,n,\ell)| \ell^{-2\sigma}.
				\]
				
				Using Weyl's bound, and that $\tau(\ell) \ll_\vep \ell^\vep$, for any $\vep>0$,
				\[
					\sum_{\ell=1}^\infty \frac{S(m,n,\ell)}{\ell^{2s}}\ll_\vep \sum_{\ell=1}^\infty \frac{1}{\ell^{2\sigma - \hf -\vep}} 
					\sum_{\substack{d |\ell\\ d|(m,n)}} d^\hf . 
				\]
				Here we overcount, each $\ell$ term with $(m,n,\ell) = d$ shows up. Exchange orders of summation and note that
				\[
					\sum_{d|(m,n)} \frac{1}{d^{2\sigma - 1 -\vep }} \leq \tau((m,n)) \ll_\vep (m,n)^\vep,
				\]
				in the region $\sigma >\frac34$. In the same region, the sum over $\ell$ converges and we obtain the lemma.
			\end{proof}
			%proof of lemma ends
			%proof of lemma ends
			
			We can have a better bound using formula \eqref{eq:linearCombination}. Initially we bound $\widetilde{Z}_{m,-n}(s)$ in the region $\Re(s)>1$. 
			%lemma
			%lemma
			\begin{lemma}
				Let $\vep>0$. One has the bound
				\[
					\widetilde{Z}_{m,-n}(s) \ll_\vep (mn)^{\qtr + \vep},
				\]
				for $\sigma >\frac34$. 
			\end{lemma}
			%lemma ends
			%lemma ends
			
			%proof of lemma
			%proof of lemma
			\begin{proof}
				Initially assume $\sigma >1$. Split the sum over $\ell$ into dyadic segments,
				\begin{align*}
					\lt|\widetilde{Z}_{m,-n}(s)\rt| 
					&= \lt|(2\pi \sqrt{mn})^{2s-1} \sum_{\ell=1}^\infty \frac{S(m,-n,\ell)}{\ell^{2s}}\exp\lt(-\frac{4\pi \sqrt{mn}}{\ell}\rt)\rt|\\
					&\leq (2\pi \sqrt{mn})^{2\sigma-1} \mspace{-15mu} \sum_{k=-\log_2(\sqrt{mn}) }^\infty 
					\sum_{2^{-(k+1)} \leq \frac{2\pi\sqrt{mn}}{\ell} < 2^{-k} }
					\frac{1}{\ell^{2\sigma-1}} \exp\lt(-\frac{1}{2^{k}}\rt).
				\end{align*}
				Here we have used the trivial bound $S(m,n,\ell) \leq \ell$. Next,
				\begin{align*}
					\lt|Z_{m,-n}(s)\rt| &\leq  (2\pi \sqrt{mn})^{2\sigma - 1} \sum_{k=0}^{\infty} \frac{2^k 2\pi \sqrt{mn} }{(2^k 2\pi\sqrt{mn})^{2\sigma -1}}
					\exp\lt(-\frac{1 }{2^k}\rt)\\
					&\phantom{\leq } +(2\pi \sqrt{mn})^{2\sigma -1} \sum_{k=0}^{\log_2(\sqrt{mn})} \frac{2^k 2\pi\sqrt{mn} }{(2^k 2\pi \sqrt{mn})^{2\sigma -1}}
					\exp\lt(-2^k\rt)\\
					&\leq 2\pi \sqrt{mn}\lt( \sum_{k=0}^\infty \frac{1}{2^{k(2\sigma - 2)}} 
					+ \sum_{k=0}^\infty \frac{1}{2^{k(2\sigma-2)}}\exp\lt(-2^k\rt) \rt).
				\end{align*}
				The sums over $k$ converge, and therefore one has 
				\[
					\widetilde{Z}_{m,-n} (s) \ll_\sigma \sqrt{mn}
				\]
				under the assumption $\sigma >1$. In fact the bound is uniform in $\sigma$ given that $\sigma > 1 + \delta$ for $\delta>0$. 
				
				Combine this result with Lemma \ref{lem:weilBound} and use Phragm\'en Lindel\"of principle on the lines $\sigma = \frac34 + \delta$ 
				and $\sigma$ fixed but large. We then get the result.
			\end{proof}
			%proof of lemma ends
			%proof of lemma ends
			If we had used the Weil bound instead of the trivial one, we would have gotten the $(mn)^\qtr$ bound without recourse to Phragm\'{e}n-Lindel\"{o}f.
			
			After a use of \eqref{eq:linearCombination} we obtain,
			\[ 
				Z_{m,-n}(s) = \sum_{k=0}^\infty \frac{2^k}{k!} \widetilde{Z}_{m,-n}\lt(s + \frac k2\rt) \ll_\vep (mn)^{\qtr + \vep} \sum_{k=0}^\infty \frac{2^k}{k!} 
				\ll (mn)^{\qtr + \vep}.
			\]
			We have proved the following proposition
			\begin{prop}
				The opposite sign Kloosterman sum zeta function can be bounded as
				\[
					Z_{m,-n}(s) \ll_\vep (mn)^{\qtr + \vep},
				\]
				where $\Re(s) = \sigma > \frac 34$. 
			\end{prop}
			
			In the next section we look at the spectral expansions to deduce the growth of these functions in other regions.

		\subsection*{The Region of Spectral Convergence}

			Given $\sigma<0$, we bound $Z_{m,-n}(s)$ using its spectral expression. However, because the continuous spectrum contains the additional terms
			$R_r^-(s)$ that involve reciprocal zeta functions, the $\Im(s) = t$ dependence will be very different in alternating strips of length $1/2$.
			
			%theorem
			%theorem
			\begin{prop} \label{prop:spectralBound}
				Assume $\Re(s)=\sigma<0$. Ignoring the $\Im(s) = t$ dependence, one has the bound,
				\[
					Z_{m,-n}(s) \ll_s (mn)^{\hf - \sigma}.
				\]
				If $R\geq 0$ is an integer and $-R - \hf < \sigma <R$, then we can also bound the $t$ dependence of the opposite sign Kloosterman sum zeta
				function as follows:
				\[
					Z_{m,-n}(s) \ll_\sigma \begin{cases} (mn)^\theta (1 + |t|)^{\sigma}  + (mn)^{\hf - \sigma }(1 + |t|)^{4\sigma -2} &\text{ if } R = 0,\\
					(mn)^\theta (1 + |t|)^\sigma + (mn)^{\hf -\sigma}(1 + |t|)^{\sigma - 2} &\text{ if } R\geq 1. \end{cases}
				\]
				And finally in the region $-R< \sigma <-R + \hf$ for $R\geq 1$ an integer
				\[
					Z_{m,-n}(s) \mspace{-4mu} \ll \mspace{-4mu} \begin{cases}(mn)^\theta (1 + |t|)^{\sigma} 
					+ (mn)^{\hf - \sigma} (1 + |t|)^{\sigma - 2 + \frac{1}{2\pi} }& \text{on RH,}\\
					             (mn)^\theta(1 + |t|)^\sigma +  (mn)^{\hf - \sigma}(1 + |t|)^{A \log\log(10 + |t|)}\mspace{-18mu} &  \text{unconditionally,}
					                \end{cases}
				\]
				where $s$ is positioned so that its distance from all $\rho/2$ where $\zeta(\rho) = 0$ is at least $1/\log(|t|)$ and $A$ is some positive constant.
			\end{prop}
			%theorem ends
			%theorem ends
			
			\begin{proof}
				From Stirling's formula,
				\begin{align*}
					Z_d(s) &\ll_\sigma \sum_{t_j} \lt| \lambda_j(n) \lambda_j(m) \rt| \lt|\rho_j(1)\rt|^2 \lt|
					\Gamma(s - \tfrac12 + it_j)\Gamma(s - \tfrac12 - it_j)\rt|\\
					&\ll (mn)^\theta \sum_{t_j} e^{-\frac{\pi}{2} (-2t_j + |t + t_j|+| t-t_j|)} 
					(1 + |t - t_j|)^{\sigma - 1} (1 + |t + t_j|)^{\sigma- 1} .
				\end{align*}
				The sum is symmetric under conjugation, so let us without loss of generality assume that $t>0$. The exponential factor is 1 
				when $t_j>t$ and is equal to $e^{\pi(t- t_j)}$ when $t_j < t$. So if $t_j$ is much less than $t$, then the 
				contribution from those terms will be negligible. Separate the summation over 
				$t_j$ into dyadic intervals. Let $T >2t > T/2$ with $T = 2^k$ for some $k$. Given $1\ll \Delta \ll T$ we will look into the regions 
				where $|t - t_j| \asymp \Delta$, the choice of $\Delta = 2^i$ with $i=1,\ldots, k$, covers $t_j \leq T$. 
				
				We also will look into the intervals $2^nT \leq |t_j| \leq 2^{n+1} T$, for $n \in \N$. Then,
				\begin{align*}
					Z_d(s) &\ll_\sigma (mn)^\theta \sum_{\substack {i = 1\\ \Delta = 2^i} }^k \sum_{|t-t_j|\asymp \Delta }  e^{-\pi(t - t_j)} (1 + T)^{\sigma - 1}
					(1 + \Delta)^{\sigma -1 }\\
					&\phantom{\ll} + (mn)^\theta \sum_{n=0}^\infty \sum_{2^nT<|t_j|<2^{n+1}T } (1 + |t + t_j|)^{\sigma - 1}(1 + |t - t_j|)^{\sigma - 1}\\
					&\ll_\sigma (mn)^\theta \sum_{\substack{i=1\\\Delta=2^i}}^k 
					e^{- \Delta } (1 + T)^\sigma (1 + \Delta)^\sigma + \sum_{n=0}^\infty (1 + 2^nT)^{2\sigma}\\
					&\ll (mn)^\theta (1 + T)^\sigma
				\end{align*}
				The integral in the continuous spectrum can be similarly bounded. In fact the result is better because the continuous spectral density is less than that 
				of the discrete spectrum and bounding the divisor sum functions in $m$ and $n$ is easier than bounding Maass form coefficients. 
				The integrand contains zeta functions in the denominator and one of them is on the edge of the critical strip, 
				and since the reciprocal zeta function can be bounded by $O(\log(1 +|t|))$, we have
				\[
					I(s) \ll (1 + |t|)^{\sigma - 1} \log(1 + |t|)
				\]
				
				The other piece of the continuous spectrum, i.e. the additional terms that show up when the function is analytically continued past the lines
				$\sigma = \hf - r$ for $r \geq 0 $ an integer are given by,
				\[
					 R_r^-(s) = 2 \frac{(-1)^r}{r!} \Gamma(2s +r - 1)
					 \frac{\sigma_{2s + 2r-1}(m)\sigma_{2s+ 2r-1}(n)}{(mn)^{s +r- \hf} \zeta^*(1-2s - 2r)\zeta^*(2-2s-2r)}.
				\]
				On the one hand, ignoring the $t$ dependence
				\[
					R_r^-(s) \ll_s (mn)^{\hf - \sigma -r}.
				\]
				This choice is maximal when $r = 0$. This yields the first part of the proposition.				
				
				On the other hand, if we care about the $t$ dependence,
				\begin{align} \label{eq:RrBound}
					R_r^-(s) &\ll_\sigma \frac{1}{r!} (mn)^{\hf - \sigma - r} \frac{\Gamma(2s + r -1)}{\Gamma(\hf-s -r) \Gamma(1 -s  -r) \zeta(1-2s - 2r) } \notag\\
					&\ll (mn)^{\hf - \sigma - r} \frac{(1 + |t|)^{4\sigma + 3r -2}}{\zeta(1-2s -2r)}.
				\end{align}
				The problem with the $t$ dependence is the existence of the zeta function in the denominator therefore we first restrict our
				attention to $-R -\hf < \sigma <-R$, the left half of the strip, where the argument of the zeta functions are in the region of 
				absolute convergence. Furthermore, the subscripts of the divisor sums have real part less than zero. Hence, the largest growth in $mn$ is
				caused by the $r = 0$ term, in which case it can be bounded by $(mn)^{ \hf - \sigma}(1 + |t|)^{4\sigma - 2}$, and 
				when $r = R$, we have the largest contribution in the $|t|$ aspect. It is $(mn)^{\hf - \sigma - R} (1 + |t|)^{4\sigma +3R -2}$, which still 
				decays in $|t|$.
				
				Using the location of $R$ and $\sigma$ we deduce the bound,
				\[
					Z_c(s) \ll \begin{cases} (1 + |t|)^{\sigma-1+\vep}  + (mn)^{\hf - \sigma }(1 + |t|)^{4\sigma -2} &\text{ if } R = 0,\\
					(1 + |t|)^{\sigma-1+\vep} + (mn)^{\hf -\sigma}(1 + |t|)^{4\sigma - 2} + (mn)(1 + |t|)^{\sigma -2} \mspace{-10mu}
					&\text{ if } R\geq 1. \end{cases}
				\]
				Here $\vep>0$ is arbitrarily small and comes from the logarithmic term in bounding $I(s)$. 
				
				Now let us assume that we are in the other half of the strip, i.e. $-R < \sigma < - R + \hf$. The only trouble is going to come from the 
				$r = R$ term, because then the argument of the zeta function in the denominator is in the critical strip. We quote (9.6.3) from 
				\cite{titchmarsh1951theory} which states that
				\[
					\log \zeta(s)  = \sum_{|t - \gamma| < 1} \log (s - \rho) + O(\log t) 
				\]
				where the sum is over critical zeros $\rho = \beta + i\gamma$ of the Riemann zeta function. Together with the restriction $|s - \rho|> 1/\log t$ and 
				density results on the zeros of the Riemann zeta function, we deduce that,
				\[
					|\log \zeta(s)| \ll \log t \log \log t.
				\]
				Hence 
				\[
					\frac{1}{\zeta(s)} \ll t^{A \log\log t}
				\]
				for some constant $A$. This yields the result.
				
				In case what happens when we assume the Riemann hypothesis, we then revert to equation (14.15.1) loc.\ cit.\ which states
				\[
					\log \zeta(s) = \sum_{|t - \gamma| < \frac{1}{\log \log t} }  \log (s -\rho) + O\lt(\frac{\log t \log \log \log t}{\log \log t}\rt)
				\]
				for $\Re(s) \geq \hf$. For $\Re(s)<\hf$, we may use the functional equation. 
			\end{proof}
			%proof ends
			%proof ends
			
			One advantage of $Z_{m,-n}(s)$ to $\widetilde{Z}_{m,-n}(s)$ is that in the spectral expansion of $\widetilde{Z}_{m,n}(s)$ there are additional terms
			$\widetilde{R}_r^-(s)$ obtained when $\sigma$ passes every integer and half integer, as opposed to just the latter. This happens because of the vanishing
			of the sum over $k$ in \eqref{eq:vanishResidualTerms}. Hence we are able to obtain a safe strip of width $\hf$ where we are away from the critical zeros of 
			$\zeta(1-2s +r)$.
			
		\subsubsection*{The Region in Between}
			
			In the region $0 <\sigma <\frac 34$, we cannot  use the spectral expansion 
			of $Z_{m,-n}(s)$ because it simply does not converge there. Hence we use \eqref{eq:linearCombination}, in order to bound the $mn$
			dependence, and for that we need the bounds on $\widetilde{Z}_{m,-n}(s)$ in that region. The proof is similar to that of Proposition \ref{prop:spectralBound}
			above. We split the spectral sum into dyadic intervals, and estimate.
 			%lemma
 			%lemma
 			\begin{lemma}
 				Let $\delta >0$. In the region $\Re(s) = \sigma >\hf + \delta$, we have the bound
 				\[
 					\widetilde{Z}_{m,-n}(s) \ll_\delta (mn)^\theta (1 + |t|)^\hf
 				\]
 				where $\theta$ is the best progress towards the Ramanujan-Petersson conjecture, and $t = \Im(s)$. 
 				
 				In the region $0<\sigma<\hf$ the function is not bounded in terms of $t$ due to the additional term $\widetilde{R}_0(s)$ in the continuous
 				spectrum. It contains the reciprocal of the Riemann zeta function. However we can still bound the $m$ and the $n$ dependence, not on the zeros
 				of $\zeta(2s)$ by
 				\[     
 					\widetilde{Z}_{m,-n}(s) \ll_s (mn)^{\hf -\sigma} +(mn)^\theta.
 				\]
 				The $t$ dependence is given as 
 				\begin{equation}
					\widetilde{Z}_{m,-n}(s) \ll \begin{cases} (mn)^{\hf - \sigma} (1 + |t|)^{4\sigma - 2 + \frac{1}{2\pi}} +(mn)^\theta(1 + |t|)^\hf &\text{on RH,} \\
					(mn)^{\hf - \sigma} t^{A \log\log (10 + |t|)} + (mn)^\theta (1 + |t|)^{\hf}&\text{unconditionally}. \end{cases}
 				\end{equation}
 				if $s$ is of distance $1/\log(1 + |t|)$ away from all zeros of $\zeta(2s)$ in the critical strip and $A>0$.
 			\end{lemma}
 			%lemma ends
 			%lemma ends
 			
			%proof of lemma
			%proof of lemma
 			\begin{proof}
 				The proof is analogous to Proposition \ref{prop:spectralBound}. Apply Stirling's formula to the Gamma functions involved,
 				\begin{align*}
 					&\widetilde{Z}_d(s) \ll_\sigma \sum_{t_j} \lambda_j(n) \lambda_j(m) \rho_j(1) \overline{\rho_j(-1)}
 					\frac{\Gamma(2s - 1 + 2it_j)\Gamma(2s - 1 - 2it_j)}{\Gamma(2s - \hf) }\\
 					&\ll (mn)^\theta \sum_{t_j} e^{\pi (t_j - |t + t_j|-| t-t_j|+ |t|)} 
 					\lt((1 + |t - t_j|) (1 + |t + t_j|)\rt)^{2\sigma- \frac 32} (1 + |t|)^{1\mspace{-2mu}-\mspace{-2mu}2\sigma}
 				\end{align*}
 				We again assume $t>0$ for convenience.. Let us separate the regions of $t_j$ into dyadic intervals.
 				First let $T >2t$ with $T = 2^k$ for some $k$. Given $1\ll \Delta \ll T$ we will look into the regions where $|t - t_j| \asymp \Delta$,
 				the choice of $\Delta = 2^i$ with $j$ up to $k$ covers $t_j \leq T$. We need to consider larger $\Delta$ for the rest of the $t_j$. Then,
 				\begin{align*}
 					\widetilde{Z}_d(s) &\ll_\sigma \sum_{\substack{i=0\\ \Delta = 2^i}}^k \sum_{|t-t_j|\asymp \Delta} 
 					e^{-\Delta} (1 + T)^{-\hf} (1 + \Delta)^{2\sigma - \frac32}\\
 					&\phantom{\ll_\sigma } +  \sum_{\substack{i=k+1\\ \Delta = 2^i}}^\infty  \sum_{|t-t_j|\asymp \Delta} 
 					e^{-\Delta} (1 + T)^{1 - 2\sigma} (1 + \Delta)^{4\sigma - 3}\\
 					&= (1 + T)^{\hf} \sum_{\substack{i =1\\ \Delta = 2^i }}^k e^{- \Delta} ( 1 + \Delta)^{2\sigma - \hf} 
 					+ (1 + T)^{1 - 2\sigma} \sum_{\substack{i=k+1\\ \Delta = 2^i}}^\infty e^{-\Delta}  (1 + \Delta)^{4\sigma - 1},
 				\end{align*}
 				where in the last line we have used Weyl's law 
 				\[
 					\#\{t_j: |t- t_j| \asymp \Delta \} \ll \begin{cases} \Delta T &\text{ for } 1 \ll \Delta \ll T, \\
 					T^2 &\text{ for } \Delta \gg T. \end{cases}				
 				\]
 				Note that the second sum goes to $0$ as $T \to \infty$. In the first sum, the sum over $\Delta$ converges to a number (depending on $\sigma$),
 				even if we complete the sum over $i$ to infinity. Thus,
 				\[
 					\widetilde{Z}_{d}(s) \ll_\sigma (mn)^\theta (1 + T)^\hf
 				\]
 				
 				The integral in the continuous spectrum is majorized by this quantity.
 				
 				The argument up to this point is equally valid for the regions $1/2 < \sigma <3/4$ and $0< \sigma <\hf$ as long as $s$ is not on the 
 				zeros of $\zeta(2s)$.
 				However, once $s$ is past the line $\sigma = \hf$, the term $\widetilde{R}_0^-(s)$ shows up. Looking at \eqref{eq:modifiedZetaResidualTerm}
 				we are able to see that
 				\[
					\widetilde{R}_0(s) \ll_s (mn)^{\hf - \sigma}
 				\]
 				in this region. Once again the greatest contribution is from the $r = 0$ term.
 				If we do not ignore the $s$ dependence, then the bound is given by
 				\[
					\widetilde{R}_r(s) \ll (mn)^{\hf - \sigma} \frac{(1 + |t|)^{4\sigma - 2}}{|\zeta(1-2s)|}.
 				\]
 				The considerations about the size of the reciprocal zeta function at the end of Proposition \ref{prop:spectralBound} apply once again
 				to give the lemma.
 			\end{proof}
 			%proof of lemma ends
 			%proof of lemma ends

 			Using this lemma and the fact that $Z_{m,-n}(s)$ can be written as an infinite linear combination, we deduce that if $\hf <\sigma < \frac34$, then
 			\begin{align*}
				Z_{m,-n}(s) &= \sum_{k = 0}^\infty \frac{2^k}{k!} Z_{m,-n}\lt(s + \frac k2\rt) \\
				&= \widetilde{Z}_{m,-n}(s) + O_\vep((mn)^{\qtr + \vep}) \\
				&\ll_\vep (mn)^\theta(1 + |t|)^\hf + (mn)^{\qtr + \vep}
			\end{align*}
 			and using Phragm\'en-Lindel\"of we get \eqref{eq:phragmenLindelof}. In the region $0<\sigma<\hf$, if we ignore the $t$ dependence,
 			\[
				Z_{m,-n}(s) = \widetilde{Z}_{m,-n}(s) + \widetilde{Z}_{m,-n}\lt(s + \hf\rt) + O_\vep((mn)^{\qtr + \vep}) 
				\ll_{\vep,s} (mn)^{\min\{\hf - \sigma, \qtr + \vep\}}.
 			\]
 			Otherwise we get \eqref{eq:firstCriticalBound}. This proves Theorem \ref{thm:growth}.

	\section{Sums of Kloosterman Sums}
	
		We are going to use the analytic properties of the zeta function $Z_{m,-n}(s)$ in order to prove a bound on
		\[
			S_{m,-n}(X) = \sum_{\ell \leq X} \frac{S(m,-n, \ell)}{\ell}. 
		\]
		We find a bound which is uniform in $m$ and $n$. 
		
		\begin{proof}[Proof of Theorem \ref{thm:sharpCutoff}]
		The proof is a standard application of Perron's formula as in \cite{davenport1967multiplicative}. Let $\sigma_1> 3/4$ and $T$ be a large quantity. 
		Let $X$ be a half integer. Set
		\[
			I(X, T) = \int_{\sigma_1 - iT}^{\sigma_1 + iT} Z_{m,-n}(s) \lt(\frac{X}{2\pi \sqrt{mn}} \rt)^{2s - 1} \frac{\d s}{s}.
		\]
		On the one hand, if express $Z_{m,-n}(s)$ as a Dirichlet series, exchange the orders of integration and summation, we obtain
		\[
			I(X,T) = S(X) + O\lt(\sum_{\ell=1}^\infty \frac{S(m,-n,\ell)}{\ell} \lt(\frac{X}{\ell}\rt)^{2\sigma_1 -1} \min\lt\{1, \frac{1}{T |\log X/\ell|} \rt\}\rt)
		\]
		Split the sum inside the big-oh into regions $\ell < 3X/4$, $3X/4 < \ell < X$, $X < \ell < 5X/4$ and $5X/4 < \ell$, calling 
		them $D_1, D_2, D_3$ and $D_4$ respectively. Analyze each term separately. Firstly,
		\[
			D_1 \leq X^{2\sigma_1-1} \sum_{\ell=1}^{\frac{3X}{4} } \frac{|S(m,-n,\ell)|}{\ell^{3/2}}  \frac{1}{T |\log(3/4)|}
			\ll \frac{X^{2\sigma -1 } \log^2 X}{T}(m,n)^\vep.
		\]
		$D_4$ can be similarly bounded. 
		
		In bounding $D_2$ we first realize that 
		\[
			\log X/\ell = - \log\lt( 1 - \frac{X - \ell}{X}\rt) \geq \frac{X - \ell }{\ell}.
		\]
		Call $\nu = X - \ell$. Then separate the sum into the regions $\nu < X/T$ and $X/T <\nu < X/4$. We can also bound these sums by the term above.
		
		Hence we have
		\begin{equation}
			I(X) = S(X) + O\lt(\frac{X^\hf \log^2 X}{T} (m,n)^\vep \rt).
		\end{equation}
		
		Let $\sigma_0>\hf$. We will move the line of integration to the vertical line of length $2T$ and abcissa $\sigma_1$. 
		
		\[
			I(X) = \lt(\int_{\sigma_0+iT}^{\sigma_1 + iT} + \int_{\sigma_1-iT}^{\sigma_0 -iT} + \int_{\sigma_0-iT}^{\sigma_0+iT} \rt) 
			Z_{m,-n}(s)\lt(\frac{X}{2\pi \sqrt{mn}} \rt)^{2s - 1} \frac{\d s}{s}.
		\]
		
		In the previous section we have established bounds for $Z_{m,-n}(s)$ in each of these domains. For the top and bottom parts of the rectangle, 
		we may use \eqref{eq:phragmenLindelof} and see that,
		\begin{align}
			&\int_{\sigma_0+iT}^{\sigma_1 + iT} + \int_{\sigma_1-iT}^{\sigma_0 -iT}Z_{m,-n}(s)\lt(\frac{X}{2\pi \sqrt{mn}} \rt)^{2s - 1} \frac{\d s}{s} \notag \\
			&\ll T^{\hf}  (mn)^{\qtr + \vep} \frac{\sqrt{mn}}{X} \int_{\sigma_0}^{\sigma_1} T^{ - 2\sigma} \lt(\frac{X}{\sqrt{mn}}\rt)^{2\sigma} \d \sigma
			\label{eq:sigmaIntegral}\\
			&\ll \frac{T^{\hf}}{X}  (mn)^{\frac34 + \vep} \frac{1}{|\log(X/(\sqrt{mn}T))| } 
			\max\lt\{ \lt(\frac{X}{\sqrt{mn}T}\rt)^{2\sigma_1},\lt(\frac{X}{\sqrt{mn}T}\rt)^{2\sigma_0} \rt\}.\notag
		\end{align}
		Note that in case $X \asymp \sqrt{mn} T$ then we can simply bound the integral in \eqref{eq:sigmaIntegral} with $O(1)$. 
		
		In the vertical line we use the bound \eqref{eq:largerThanHalfBound}, 
		
		\begin{align*}
			\int_{\sigma_0-iT}^{\sigma_0+iT} &Z_{m,-n}(s)\lt(\frac{X}{2\pi \sqrt{mn}} \rt)^{2s - 1} \frac{\d s}{s} \\
			&\ll \lt(\frac{X}{\sqrt{mn}}\rt)^{2\sigma_0-1}\lt(\int_1^T \frac{(mn)^\theta}{t^\hf}  +  \frac{(mn)^{\qtr + \vep}}{t} \d t\rt)\\
			&\ll X^{2\sigma_0-1}\lt((mn)^{\theta-(\sigma_0 - \hf)}T^\hf  +  (mn)^{\qtr + \vep - (\sigma_0-\hf)} \log T\rt).
		\end{align*}
		
		There are three cases, $X\leq \hf \sqrt{mn}T$, $X \geq 2 \sqrt{mn}T$ and $X$ in between. In the second case, choose $T = X^{\frac13}$, in the other cases choose 
		$T = (mn)^{1/4}$.
		
		Now choose $\sigma_0 = \hf  + \vep$ and $\sigma_1 = \frac34 + \vep$. Combining everything we deduce that,
		\[
			S(X) \ll_\vep X^{\frac{1}{6} +\vep}  ((m,n)^\vep + (mn)^{\theta + \vep})  + X^\vep (mn)^{\qtr + \vep }.\qedhere
		\]
		\end{proof}

\bibliographystyle{alpha}
\bibliography{/u/ekiral/Documents/MehmetBib}

\end{document}